\documentclass[12pt,a4paper]{amsart}
\usepackage{amsmath}
\usepackage{amssymb}
\usepackage{tikz}
\usetikzlibrary{decorations.pathreplacing}
\usepackage[colorlinks=true]{hyperref}
\usepackage[a4paper,portrait]{geometry}
\usepackage{color}
\usepackage{longtable}
\usepackage{caption}
\usepackage{subcaption}
\usepackage[abbrev,nobysame,non-sorted-cites]{amsrefs}
\usepackage{bm}
\usepackage{comment}

\allowdisplaybreaks[3]

\newcommand{\Z}{\mathbb{Z}}
\newcommand{\Q}{\mathbb{Q}}

\DeclareMathOperator{\Gal}{Gal}

\newtheorem{theorem}{Theorem}
\newtheorem*{theorem*}{Theorem}

\newtheorem{lemma}[theorem]{Lemma}

\newtheorem*{corollary*}{Corollary}

\newtheorem{conjecture}[theorem]{Conjecture}

\theoremstyle{definition}

\theoremstyle{remark}

\newtheorem{remark}[theorem]{Remark}
\newtheorem{question}[theorem]{Question}

\begin{document}

\author[Xuejun Guo]{Xuejun Guo}

\address{School of Mathematics, Nanjing University, Nanjing 210093, People's Republic of China}

\email{guoxj@nju.edu.cn}

\author[Zhengyu Tao]{Zhengyu Tao}

\address{School of Mathematics, Hefei University of Technology, Hefei 230009, People's Republic of China}

\email{taozhy@hfut.edu.cn}

\title[Counterexamples to Stanley's conjecture on dimer coverings]{Counterexamples to Stanley's conjecture on dimer coverings}

\date{\today}

\subjclass[2020]{Primary 05A15, 05B45; Secondary 11B39, 11R18}

\keywords{dimer coverings, Stanley's conjecture, Pell numbers}

\thanks{}

\begin{abstract}
Let $Q_k(x)$ be Stanley's explicit denominator for the dimer-covering generating function $F_k(x)=\sum_{n\ge0}A_{k,n}x^n$ of $k\times n$ rectangles. Stanley conjectured in 1985 that $Q_k(x)$ has only simple roots; this longstanding conjecture was recently recorded in Lai's list of open problems on tilings (see~\cite{Lai24}*{Problem 33}). We disprove the conjecture by proving that $Q_{14h-1}(x)$ and $Q_{30h-1}(x)$ have repeated roots for every $h\ge1$; in particular, $k=13$ is the smallest counterexample. The construction comes from two exceptional multiplicative identities among trigonometric algebraic units. We further propose a conjecture concerning this class of trigonometric identities, which appears to be related to Robinson's problem on primitive Pell factors.
\end{abstract}

\maketitle

\section{Introduction} \label{intro}

Let $k$ be a positive integer and $A_{k,n}$ denote the number of ways to tile a $k\times n$ rectangle with $1\times 2$ rectangles (usually called dimers or dominoes). For a fixed $k$, it is a well-known fact that the generating function
\[
  F_k(x)=\sum_{n\geq 0} A_{k,n}x^n
\]
of $\{A_{k,n}\}_{n\ge 0}$ can be written as a quotient of two polynomials with integer coefficients, namely, a rational function (see, for example, \cite{KP80}).

In~\cite{Sta85}, Stanley used the exact product formula for $A_{k,n}$ by Kasteleyn, Temperley-Fisher, and Lieb \cites{Kas61, TF61,Lie67} to explicitly compute a pair of polynomials $P_k(x),Q_k(x)\in\Z[x]$ for each $k$ such that $F_k(x)=P_k(x)/Q_k(x)$ and $Q_k(0)=1$. We record Stanley's formula for the denominator as follows: let $\ell=\left\lfloor\frac{k}{2}\right\rfloor,$
\begin{equation}\label{eq:cjej}
  c_j=\cos\frac{j\pi}{k+1}+\sqrt{1+\cos^2\frac{j\pi}{k+1}},
  \qquad
  \bar{c}_j=\cos\frac{j\pi}{k+1}-\sqrt{1+\cos^2\frac{j\pi}{k+1}}.
\end{equation}
Clearly, we have $c_j\bar{c}_j=-1$. For each subset $S$ of $\{1,\cdots,\ell\}$, define
\begin{equation}\label{eq:aS}
  a_S=\Big(\prod_{j\in S}c_j\Big)\Big(\prod_{j\in\{1,\dots,\ell\}\setminus S}\bar{c}_j\Big),
\end{equation}
where an empty product, if it occurs, is understood to be \(1\).
 Then
\begin{equation}\label{eq:defofQk}
  Q_k(x):=
  \begin{cases}
  \displaystyle\prod_{S\subseteq \{1,\ldots,\ell\}}(1-a_S x), & k \text{ even},\\[1.2em]
  \displaystyle\prod_{S\subseteq \{1,\ldots,\ell\}}(1-a_S^2 x^2), & k \text{ odd}.
\end{cases}
\end{equation}
At the end of \cite{Sta85}, Stanley stated the following conjecture:
\begin{conjecture}[Stanley]\label{conj:Stanley}
The polynomial $Q_k(x)$ has distinct roots.
\end{conjecture}

\begin{remark}\label{rem:Stanley}
Stanley pointed out in~\cite{Sta85} that Conjecture~\ref{conj:Stanley} is equivalent to the assertion that
\(2^{\lfloor (k+1)/2\rfloor}\) is the minimum order of a linear recurrence satisfied by \(A_{k,n}\). This is also equivalent to the assertion that the polynomials \(P_k(x)\) and \(Q_k(x)\) are relatively prime.
 Stanley also proved in~\cite{Sta85} that if $F_k(x)=P_k(x)/Q_k(x)$ is reduced to lowest term, then the denominator has distinct roots. Conjecture~\ref{conj:Stanley} is also collected in~\cite{Lai24} as Problem 33.
\end{remark}

The purpose of this note is to show that Stanley's conjecture, in the form stated above, is false. The smallest counterexample is $k=13$. In fact, we can construct two infinite families of  counterexamples to Conjecture~\ref{conj:Stanley} (see Theorem~\ref{thm:mainresult}). This note is organized as follows. In Section~\ref{sec:mainresult}, we first establish a simple criterion (Lemma~\ref{lem:criterion}) for the polynomial $Q_k(x)$ to have repeated roots. Then, with the help of two elegant trigonometric identities, we prove Theorem~\ref{thm:mainresult}. In Section~\ref{sec:finalremarks}, we revise Stanley's original conjecture to Conjecture~\ref{conj:revised_conj}. Some connections with Robinson's conjecture are also discussed.

\section{Main results}\label{sec:mainresult}

Let $k\ge 1,\ell=\left\lfloor\frac{k}{2}\right\rfloor$, and let $c_j,\bar{c}_j$ be defined as in \eqref{eq:cjej}. For each subset $S$ of $\{1,\ldots,\ell\}$, define $a_S$ as in~\eqref{eq:aS} and set $b_S=\prod_{j\in S}c_j$. Since $c_j\bar{c}_j=-1$, we have $\displaystyle a_S=(-1)^{\ell-|S|}\frac{b_S^2}{b_{\{1,\dots,\ell\}}}$. According to the definition~\eqref{eq:defofQk} of $Q_k(x)$, we clearly have:
\begin{lemma}\label{lem:criterion}
 With the notation above:
  \begin{itemize}
		\item [(i)]  if $k$ is odd, then $Q_k(x)$ has repeated roots if and only if there exist  distinct subsets $S,T$ of $\{1,\dots,\ell\}$ such that $b_S=b_T$;
    \item [(ii)] if $k$ is even, then $Q_k(x)$ has repeated roots if and only if there exist  distinct subsets $S,T$ of $\{1,\dots,\ell\}$ such that $b_S=b_T$ and $|S|\equiv|T|\pmod 2$.
	\end{itemize}
\end{lemma}

The two families of counterexamples to Conjecture~\ref{conj:Stanley} are:

\begin{theorem}\label{thm:mainresult}
For every integer $h\geq 1$, the polynomials $Q_{14h-1}(x)$ and $Q_{30h-1}(x)$ have repeated roots.
\end{theorem}
\begin{proof} 
By Lemma \ref{lem:criterion}, Theorem~\ref{thm:mainresult} follows immediately from the following two identities:
\begin{equation}\label{eq:keyeq1}
  \cos\frac{\pi}{7}+\sqrt{1+\cos^2\frac{\pi}{7}}
  =
  \left(\cos\frac{2\pi}{7}+\sqrt{1+\cos^2\frac{2\pi}{7}}\right)
  \left(\cos\frac{3\pi}{7}+\sqrt{1+\cos^2\frac{3\pi}{7}}\right),
\end{equation}
\begin{equation}\label{eq:keyeq2}
  \cos\frac{7\pi}{30}+\sqrt{1+\cos^2\frac{7\pi}{30}}
  =
  \left(\cos\frac{\pi}{3}+\sqrt{1+\cos^2\frac{\pi}{3}}\right)
  \left(\cos\frac{13\pi}{30}+\sqrt{1+\cos^2\frac{13\pi}{30}}\right).
\end{equation}
Indeed, if $k=14h-1$, then \eqref{eq:keyeq1} shows that $b_{\{2h\}}=b_{\{4h,6h\}}$; if $k=30h-1$ then \eqref{eq:keyeq2} shows that $b_{\{7h\}}=b_{\{10h,13h\}}$.
\end{proof}

\begin{remark}
Although $Q_k(x)$ is generally difficult to compute (the degree of $Q_{k}(x)$ is $2^{\lfloor(k+1)/2\rfloor}$), we can still use a computer algebra system to compute $Q_k(x)$ for some small $k$. 
One can check that the polynomials \(Q_k(x)\), \(1\le k\le 12\), all have distinct roots. Note that  $Q_1(x),\dots,Q_8(x)$ are computed by Klarner and Pollack in~\cite{KP80}. So, according to Theorem~\ref{thm:mainresult}, $k=13$ is the smallest counterexample to Conjecture~\ref{conj:Stanley}. In fact,

\[\begin{split}
  Q_{13}(x) & = \big(x^{16}-71 x^{14}+952 x^{12}-3976 x^{10}+6384 x^8-3976 x^6+952 x^4-71 x^2+1\big)^2 \\
  & \quad\  \cdot \big(x^{48}-1846 x^{46}+684333 x^{44}-88863671 x^{42}+5304620048x^{40}-165441761576 x^{38} \\
  & \quad\ +2911114622304 x^{36}-30365738521053 x^{34}+194344571749094 x^{32} \\ 
  & \quad\ -781085479259969 x^{30} +2002212789950035 x^{28}-3316160898776544x^{26} \\ 
  & \quad\ +3593291750966064 x^{24}-2571925079697792 x^{22} +1222383831824259 x^{20} \\
  & \quad\ -385896704246482 x^{18}+80456527547383 x^{16}-10928486271989 x^{14} \\
  & \quad\ +945014295568x^{12}-50346231208 x^{10}+1585650976 x^8-27716767 x^6+242450 x^4 \\ 
  & \quad\ -923 x^2+1\big) \cdot \big(x^{48}-923 x^{46}+242450 x^{44}-27716767 x^{42}+1585650976 x^{40} \\
  & \quad\ -50346231208x^{38}+945014295568 x^{36} -10928486271989 x^{34}+80456527547383 x^{32} \\
  & \quad\ -385896704246482 x^{30}+1222383831824259 x^{28} -2571925079697792 x^{26} \\
  & \quad\ +3593291750966064x^{24}-3316160898776544 x^{22}+2002212789950035 x^{20} \\
  & \quad\ -781085479259969 x^{18}+194344571749094 x^{16}-30365738521053 x^{14} \\
  & \quad\ +2911114622304 x^{12} -165441761576x^{10}+5304620048 x^8-88863671 x^6+684333 x^4 \\
  & \quad\ -1846 x^2+1\big).
\end{split}\]
We can also compute $P_{13}(x)$ by multiplying $Q_{13}(x)$ with the generating series of $\{A_{k,n}\}_{n\ge 0}$. It turns out that
\[\gcd(P_{13}(x),Q_{13}(x))=x^{16}-71 x^{14}+952 x^{12}-3976 x^{10}+6384 x^8-3976 x^6+952 x^4-71 x^2+1.\]
So, after cancellation, the denominator of $F_{13}(x)=P_{13}(x)/Q_{13}(x)$ indeed has distinct roots (see Remark~\ref{rem:Stanley}).
\end{remark}

In the last paragraph of~\cite{Sta85}, Stanley further reported a result of J.~Lagarias: $Q_k(x)$ is irreducible whenever $k+1$ is an odd prime different from $7$. However, Stanley did not include Lagarias's proof, nor did he cite a separate published account of it. For the sake of completeness, we provide a proof of this result in the following, no originality is claimed for the statement itself.

\begin{theorem}[Lagarias]\label{thm:irrQk}
If $k+1$ is an odd prime other than $7$, then Stanley's polynomial $Q_k(x)$ is irreducible over $\mathbb{Q}$.
\end{theorem}

\begin{proof}
Let $p=k+1,\ell=\lfloor\frac{k}{2}\rfloor,K=\Q(\cos\frac{\pi}{p})$, and $M$ be the splitting field of $Q_k(x)$ over $\Q$. Recall that $Q_k(x)=\prod_{S\subseteq \{1,\ldots,\ell\}}(1-a_S x)$, where $a_S$ is given by~\eqref{eq:aS}. Since
\[
\frac{a_{\{1,\dots,\ell\}}}{a_{\{2,\dots,\ell\}}}+\frac{a_{\{2,\dots,\ell\}}}{a_{\{1,\dots,\ell\}}} =\frac{\cos\frac{\pi}{p}+\sqrt{1+\cos^2\frac{\pi}{p}}}{\cos\frac{\pi}{p}-\sqrt{1+\cos^2\frac{\pi}{p}}}+\frac{\cos\frac{\pi}{p}-\sqrt{1+\cos^2\frac{\pi}{p}}}{\cos\frac{\pi}{p}+\sqrt{1+\cos^2\frac{\pi}{p}}} = -4-2\cos\frac{2\pi}{p},
\]
we have $K=\Q(\cos\frac{2\pi}{p})\subseteq M\subseteq L:=K(\sqrt{d_1},\dots,\sqrt{d_\ell})$, where $d_j=4(1+\cos^2\frac{j\pi}{p}),1\le j\le\ell$. To prove that $Q_k(x)$ is irreducible over $\Q$, it suffices to prove that $\Gal(L/K)$ acts transitively on $\{a_S\}_{S\subseteq\{1,\dots,\ell\}}$, which is equivalent to proving that $[L:K]=2^\ell$. Since if $[L:K]=2^\ell$, then for each subset $S$ of $\{1,\dots,\ell\}$, there exists $\sigma_S\in\Gal(L/K)$ such that
\[\sigma_S(\sqrt{d_j})=\begin{cases} \sqrt{d_j}, &\text{if }j\in S,\\ -\sqrt{d_j} &\text{if }j\in\{1,\dots,\ell\}\setminus S, \end{cases}\]
we clearly have $\sigma(a_{\{1,\dots,\ell\}})=a_S$. According to elementary Kummer theory, proving that $[L:K]=2^\ell$ is equivalent to proving that if there exist integers $e_1,\dots,e_\ell$ such that
\begin{equation}\label{eq:product_square}
\prod_{j=1}^{\ell}d_j^{e_j}\in K^{\times 2},
\end{equation}
then we have $e_j\equiv 0\pmod 2$ for each $j=1,\dots,\ell$. In the following, we will prove the latter.

Let $\zeta=e^{\frac{2\pi i}{p}}$. We have $K=\Q(\zeta+\zeta^{-1})$,
\[d_j=6+\zeta^j+\zeta^{-j}=\zeta^{-j}(\zeta^j+3+2\sqrt{2})(\zeta^j+3-2\sqrt{2}).\]
Thus,
\[\begin{split}N_{K/\Q}(d_j)^2 & = \prod_{j=1}^{p-1}(6+\zeta^j+\zeta^{-j}) \\ & = \prod_{j=1}^{p-1} (\zeta^j+(3+2\sqrt{2}))\prod_{j=1}^{p-1} (\zeta^j+(3-2\sqrt{2})) \\ & = \left(\frac{(1+\sqrt2)^p-(1-\sqrt2)^p}{2\sqrt2}\right)^2.\end{split}\]
Since the $p$-th Pell number $u_p=\frac{(1+\sqrt2)^p-(1-\sqrt2)^p}{2\sqrt2}$ is not divisible by $p$, we must have $p\nmid N_{K/\Q}(d_j)$. This non-divisibility condition immediately implies that the ideals $(d_1),\dots,(d_\ell)$ are pairwise coprime in $\mathcal{O}_K$. Indeed, suppose that a prime ideal $\mathfrak{q}$ of $\mathcal{O}_K$ divides both $(d_i)$ and $(d_j)$ for some distinct $i,j\in\{1,\dots,\ell\}$, then it must also divide the ideal
\[(d_i-d_j)=(\zeta^i+\zeta^{-i}-\zeta^j-\zeta^{-j})=(1-\zeta^{i-j})(1-\zeta^{i+j})=\mathfrak{p}^2,\]
where $\mathfrak{p}=(1-\zeta)$ is the only prime ideal of $\mathcal{O}_K$ lying above $(p)$, then we have $\mathfrak{q}=\mathfrak{p}$ and $\mathfrak{p}\mid(d_j)$, this implies $p\mid N_{K/\Q}(d_j)$, a contradiction!

We also need the fact that $u_p$ is not a square. Let $r_n=\frac{(1+\sqrt2)^n+(1-\sqrt2)^n}{2}$. Suppose that $u_p=y^2$. Then $r_p^2-2y^4=r_p^2-2u_p^2=-1$, namely $r_p^2+1=2y^4$. Note that the only positive integer solutions to Ljunggren's equation
\[X^2+1=2Y^4\]
are $(X,Y)=(1,1)$ and $(239,13)$ (see, for example~\cites{Lju42, Che94, Dra07}), we have $u_p=1$ or $169$. Since the Pell numbers $u_n$ are strictly increasing for $n\ge 1$ and $u_7=169$, we know that for our $p\neq7,u_p$ is never a square. 

Finally, assume that there exist $e_1,\dots,e_\ell\in\{0,1\}$ such that~\eqref{eq:product_square} holds. If $e_i=1$ for some $i$, then because $N_{K/\mathbb Q}(d_i)=u_p$ is not a square, the ideal $(d_i)$ has some prime ideal divisor
\(\mathfrak q\) with odd valuation: $v_{\mathfrak q}(d_i)\equiv1\pmod2$. Since the ideals $(d_1),\dots,(d_\ell)$ are pairwise coprime, $\mathfrak q$ divides no $(d_j)$ with $j\ne i$. Hence
\[
        v_{\mathfrak q}\left(\prod_{j=1}^{\ell}d_j^{e_j}\right)
        =
        v_{\mathfrak q}(d_i)
        \equiv1\pmod2,
\]
which is impossible for a square in $K^\times$. Therefore, we have $e_j=0$ for all $j$.
\end{proof}

\begin{remark}
If $k+1$ is neither prime nor a multiple of $7$ or $30$, $Q_k(x)$ may be irreducible, for example
\[
\begin{split}
Q_{8}(x) & = x^{16}-x^{15}-76 x^{14}-69 x^{13}+921 x^{12}+584 x^{11}-4019 x^{10}-829 x^9 \\ & \quad +7012 x^8 -829 x^7-4019 x^6+584 x^5+921 x^4-69x^3-76 x^2-x+1;
\end{split}
\]
or reducible, for example
\[\begin{split}Q_7(x) & = (x^8-2 x^7-26 x^6-22 x^5+42 x^4+22 x^3-26 x^2+2 x+1) \\ & \quad  \cdot (x^8+2 x^7-26 x^6+22 x^5+42 x^4-22 x^3-26 x^2-2x+1).\end{split}\]
\end{remark}

\section{Further discussions} \label{sec:finalremarks}

The reader may observe that the counterexamples constructed in Theorem~\ref{thm:mainresult} rely heavily on the elementary multiplicative identities \eqref{eq:keyeq1} and \eqref{eq:keyeq2}. It is not hard to see that, according to Lemma~\ref{lem:criterion}, the question of classifying all $k$ such that $Q_k(x)$ has repeated roots is equivalent to the following question:

\begin{question}\label{question:classification}
For which integers $k\geq 1$ do there exist distinct subsets $S,T\subseteq \{1,\dots,\lfloor k/2\rfloor\}$ such that
\begin{equation}\label{eq:subsetproduct}
  \prod_{j\in S}\left(\cos\frac{j\pi}{k+1}+\sqrt{1+\cos^2\frac{j\pi}{k+1}}\right)
  =
  \prod_{j\in T}\left(\cos\frac{j\pi}{k+1}+\sqrt{1+\cos^2\frac{j\pi}{k+1}}\right)?
\end{equation}
When $k$ is even, which of these identities have $|S|\equiv |T|\pmod 2$?
\end{question}

Based on our numerical experiments, for $k\le 50$, essentially, there are only $4$ formulas of the form~\eqref{eq:subsetproduct}, they are: \eqref{eq:keyeq1}, \eqref{eq:keyeq2}, and
\begin{equation}\label{eq:keyeq3}
  \cos\frac{\pi}{30}+\sqrt{1+\cos^2\frac{\pi}{30}}
  =
  \left(\cos\frac{\pi}{3}+\sqrt{1+\cos^2\frac{\pi}{3}}\right)
  \left(\cos\frac{11\pi}{30}+\sqrt{1+\cos^2\frac{11\pi}{30}}\right),
\end{equation}
\begin{equation}\label{eq:keyeq4}
   \begin{split}
   \left(\cos\frac{\pi}{30}+\sqrt{1+\cos^2\frac{\pi}{30}}\right) & \left(\cos\frac{13\pi}{30}+\sqrt{1+\cos^2\frac{13\pi}{30}}\right) \\
  & =
  \left(\cos\frac{7\pi}{30}+\sqrt{1+\cos^2\frac{7\pi}{30}}\right)
  \left(\cos\frac{11\pi}{30}+\sqrt{1+\cos^2\frac{11\pi}{30}}\right)
\end{split}
\end{equation}
(note that any two of \eqref{eq:keyeq2}, \eqref{eq:keyeq3}, and \eqref{eq:keyeq4} imply the third). This suggests the following conjecture:
\begin{conjecture}\label{conj:tri_iden}
Let $k\ge 1$. If there exist distinct subsets
$S,T\subseteq \{1,\dots,\lfloor k/2\rfloor\}$
for which~\eqref{eq:subsetproduct} holds, then either $7\mid k+1$ or $30\mid k+1$. Moreover, if $S\cap T=\varnothing$ and there do not exist subsets $S'$ of $S$ and $T'$ of $T$ such that $\varnothing\neq S'\cup T'\neq S\cup T$ and
\[
\prod_{j\in S'}\left(\cos\frac{j\pi}{k+1}+\sqrt{1+\cos^2\frac{j\pi}{k+1}}\right)
=
\prod_{j\in T'}\left(\cos\frac{j\pi}{k+1}+\sqrt{1+\cos^2\frac{j\pi}{k+1}}\right),
\]
then the identity must be one of
\eqref{eq:keyeq1}, \eqref{eq:keyeq2}, \eqref{eq:keyeq3}, and \eqref{eq:keyeq4}.

\end{conjecture}
The above conjecture clearly implies:
\begin{conjecture}
\label{conj:revised_conj}
Stanley's polynomial \(Q_k(x)\) has repeated roots if and
only if
\[
   k=14h-1
   \qquad\text{or}\qquad
   k=30h-1
\]
for some positive integer \(h\).
\end{conjecture}

It is interesting that the same two integers, $7$ and $30$, also occur in a classical  open problem in number theory. Let
\[L_n=\prod_{\substack{1\le k\le n-1 \\ (n,k)=1}}\big((1+\sqrt{2})-(1-\sqrt{2})e^{\frac{2\pi ik}{n}}\big)\]
be the Sylvester-Pell cyclotomic sequence. It satisfies $u_n=\prod_{d\mid n}L_d$, where $u_n=\frac{(1+\sqrt2)^n-(1-\sqrt2)^n}{2\sqrt2}$ is the Pell sequence used in the proof of Theorem~\ref{thm:irrQk}, so $L_n$ is also called ``the primitive parts" of Pell numbers. R. K. Guy reports that R. Robinson observed that $L_7=13^2$ and $L_{30}=31^2$ and asked whether there are any further square values of $L_n$; see~\cite{Guy94}*{p.18}. There are no others in the first $10000$ terms; see~\cite{OEIS_A008555}. We will explain a connection between Robinson's conjecture and Conjecture~\ref{conj:tri_iden} in the following.

For $n\geq 2$ and $1\leq j<n$, define $c_n(j) = \cos\frac{j\pi}{n} + \sqrt{1+\cos^2\frac{j\pi}{n}}$. Then we have $L_n=R_n^2$, where
\[R_n = \prod_{\substack{1\leq k <n/2\\ (k,n)=1}}\left(c_n(k)+c_n(k)^{-1}\right).\]
So $L_n$ is a perfect square if and only if $R_n$ is an integer. Note that the exceptional identities \eqref{eq:keyeq1}, \eqref{eq:keyeq2}, \eqref{eq:keyeq3} and \eqref{eq:keyeq4} are precisely $c_7(1)=c_7(2)c_7(3), c_{30}(7)=c_{30}(10)c_{30}(13),c_{30}(1)=c_{30}(10)c_{30}(11)$ and $c_{30}(1)c_{30}(13)=c_{30}(7)c_{30}(11)$,  respectively. These relations are important for computing the following two identities:
\[
\begin{split}
R_7 & = \prod_{k=1}^{3} \left(c_7(k)+c_7(k)^{-1}\right) = 13, \\
R_{30} & = \prod_{k\in\{1,7,11,13\}} \left(c_{30}(k)+c_{30}(k)^{-1}\right) = 31.
\end{split}
\]

\section*{Acknowledgements}
The first author is supported by NSFC 12231009. The second author is supported by the Natural Science Foundation of Anhui Province (Grant No. 2508085QA017).


\end{document}